\documentclass[a4paper,11pt]{article}
\usepackage{amssymb}
\usepackage{amsthm}
\newtheorem{thm}{Theorem}[section]

\newtheorem{lem}[thm]{Lemma}

\theoremstyle{definition}
\newtheorem{defn}{Definition}[section]

 \theoremstyle{remark}
\newtheorem{rmk}{Remark}

\begin{document}
\title{The isochronous center on center manifolds for three
dimensional differential systems}
\author{{ Qinlong Wang$^{1}$}\hspace{1cm}{Wentao Huang$^{2*}$}\hspace{1cm}{Chaoxiong Du$^{3}$}\\
 {\small (1\,\,School of Mathematics and
Computational Science, Guilin University} \\{\small  of
Electronic Technology, Guilin, Guangxi 541004, P.R. China)}\\
 {\small (2 Guangxi Colleges and Universities Key Laboratory of Unmanned Aerial Vehicle Telemetry,}\\{\small Guilin University of Aerospace Technology, Guilin 541004, China)} \\
{\small (3\,\,Department of Mathematics, Changsha Normal
University, Changsha,  410100, P.R. China)}}
 \footnotetext{Corresponding author: Wentao Huang (huangwentao@163.com) and Qinlong Wang (wqinlong@163.com)}
\date{}
\maketitle
\begin{center}
\begin{minipage}{130mm}
   {\small \bf Abstract.} {\small
In this paper, we give a direct method to study the isochronous centers on center manifolds of
three dimensional polynomial differential systems. Firstly, the isochronous constants of the three dimensional system are defined and its recursive formulas are given. The conditions of
the isochronous center are determined by the computation of
isochronous constants in which it doesn't need compute center manifolds of
three dimensional systems. Then the isochronous center conditions of two specific systems are discussed as
the applications of our method. The method  is
an extension and development of the formal series method for
the fine focus of planar differential systems and also readily
done with using computer algebra system
such as Mathematica or Maple.} 
\end{minipage}
\end{center}

\section{Introduction}
\hskip\parindent
 This paper is concerned with isochronous center on center manifolds for
the following three-dimensional (3D) nonlinear dynamical systems
\begin{equation}\label{3w-1}
\begin{array}{l}
  \frac{{\rm{d}}x} {{\rm{d}}t}=  - y + \sum\limits_{k + j +
l = 2}^\infty  {A_{kjl} } x^k y^j u^l  = X(x, y, u), \\
\frac{{\rm{d}}y} {{\rm{d}}t} = x + \sum\limits_{k + j + l =
2}^\infty {B_{kjl}
x^k y^j u^l }  = Y(x, y, u),\\
 \frac{{\rm{d}}u} {{\rm{d}}t} =  - du + \sum\limits_{k + j + l = 2}^\infty
{d_{kjl} } z^k w^j u^l  = U(x, y, u)
 \end{array}
 \end{equation}
where $x, y, u, t, d, A_{kjl}, B_{kjl}, d_{kjl} \in {\mathbb R}\,
(k,j,l\in {\mathbb N})$.

When the isochronous problem is restricted to planar differential
systems, the generic form
\begin{equation}\label{S-real-0}
  \begin{array}{l}
     \frac{dx}{dt}=- y + \sum\limits_{k + j = 2}^\infty  {A_{kj} } x^k y^j= X(x, y), \\
   \frac{dy}{dt}=x + \sum\limits_{k + j =
2}^\infty {B_{kj} x^k y^j }  = Y(x,y)
  \end{array}
\end{equation}
where  $x, y, t, A_{kj}, B_{kj}\in {\mathbb R}\, (k,j\in {\mathbb
N})$, or the more frequent form in the current study is considered
as follows
\begin{equation}\label{S-real}
  \begin{array}{l}
 \frac{dx}{dt}=-y+P(x,y),\,\, \frac{dy}{dt}=x+Q(x,y)
  \end{array}
\end{equation}
where $P$ and $Q$ are polynomials in $x,y$.  As is known for the
planar system, a center  is isochronous if the periods are
constant for all periodic solutions in a neighborhood of it. For
systems (\ref{S-real}) some valuable works have been done, for
example, several cases of $P$ and $Q$ as the homogeneous
polynomials with same degree can be seen in \cite{loud} for the
quadratic, \cite{pleshkan} for the cubic, \cite{chav fourth} for
the quartic and  \cite{chav fifth} for the quintic, and more other
polynomial systems studied about that problem can also be seen in
\cite{Lloyd infity,christopher,mard}. The isochronous center of non-polynomial system can be see \cite{huang09}. Usually, we should find out
the center before isochronous center is determined, namely the
center-focus determination is firstly made by computing focus
quantities, about this there are many classic or effective methods
for the planar
 systems (\ref{S-real}), for examples, given respectively by Liu et al. \cite{liuli1990}, Gine et al. \cite{Gine}, Lloyd et al. \cite{Lloyd2002}, Yu et al. \cite{Yupei2007}, and Romanovski et al. \cite{Romanovski}.


Furthermore, in order to obtain the isochronous conditions for
systems (\ref{S-real}), the classic methods of computing period
constants or isochronous constants have been given, see e.g.
\cite{chav fourth,chav fifth,gasull}.  In general, as was pointed
in \cite{christopher},
 the computation
is very difficult.  Here we would like to mention particularly the
algorithm to compute complex period constants given by the authors
of \cite{Liu-Huang} in 2003, which is readily done with using
computer algebra system such as Mathematica or Maple,  due to its
linear recurrence without complex integrating operations and
solving equations. And more the algorithm was generalized in the
planar systems with $p:-q$ arbitrary integer resonant type in
\cite{wang} for the general isochronous problem. In this paper, we
generalize and develop further the algorithm in the 3D systems
(\ref{3w-1}) for the isochronous problem on center manifolds,
which can be applied to directly figure out the necessary
condition of isochronicity on the local center manifold without
determining center conditions, though there have also existed many
classic or effective methods about that on center manifold,
for examples, given respectively by Hassard et al. \cite{Hassard}, Edneral et al. \cite{Edneral}, Yu et al. \cite{yupei-13}, and Wang et al. \cite{Wang-2010}.

The paper is organized as follows. The preliminary Section 2
provides, among other things, the definitions of generalized
isochronous center and generalized period constants on center
manifolds. In section 3, a recursive algorithm to compute
generalized period constants of the systems (\ref{3w-1}) is given.
As the applications of the algorithm, in Section 4, using our
method, we discuss the isochronous center conditions for two
specific systems. As far as we known, it is the first time to discuss the isochronous center on center manifolds of three dimensional systems.


\section{Preliminary}
\hskip\parindent
 Let us first list some isochronous definitions and preliminary results for system
(\ref{S-real-0}), then give the definition of isochronous center
for system (\ref{3w-1}).

 By means of transformation
\begin{equation}\label{2w-t1}
\begin{array}{l}
 z=x+y{\bf{i}},\; w=x+y{\bf{i}},\; t =  -
T{\bf{i}},\; {\bf{i}} = \sqrt {- 1}
\end{array}
 \end{equation}
 system
(\ref{S-real-0}) can be transformed into

\begin{equation}\label{S-complex}
\begin{array}{l}
{{{\rm{d}}z} \over {{\rm{d}}T}} = z + \sum\limits_{k + j=
2}^\infty  {a_{kj} } z^k w^j =
Z(z, w), \\
{{{\rm{d}}w} \over {{\rm{d}}T}} =  - w - \sum\limits_{k + j=
2}^\infty {b_{kj} } w^k z^j=  -W(z, w)\;
 \end{array}
 \end{equation}
where $a_{kj}, b_{kj}\in {\mathbb C}$ and $a_{kj}=\bar b_{kj}\,
(k,j\in {\mathbb N})$, namely  system (\ref{S-complex}) is called
the concomitant system of (\ref{S-real-0}).

\begin{lem}[see \cite{Russian}]\label{lem A}
 For system
(\ref{S-complex}), we can derive uniquely the following formal
series:
\begin{equation}\label{Russian1}
\xi  = z + \sum\limits_{k + j = 2}^\infty  {c_{kj} z^k w^j } ,
\eta = w + \sum\limits_{k + j = 2}^\infty  {d_{k,j} w^k z^j } ,
\end{equation}
where $c_{k+1,k}=d_{k+1,k}=0,\: k=1,2,\cdots,$ such that
\begin{equation}\label{Russian2}
\frac{{d\xi }}{{dT}} = \xi  + \sum\limits_{j = 1}^\infty  {p_j\,
\xi ^{j + 1} \eta ^j } ,\frac{{d\eta }}{{dT}} =  - \eta  -
\sum\limits_{j = 1}^\infty  {q_j \,\eta ^{j + 1} \xi ^j }.
\end{equation}
\end{lem}

We write $\mu _k  = p_k  - q_k ,\tau _k =p_k + q_k ,k = 1,2,3,
\cdots $, then we have

\begin{defn}[see \cite{liuli1989}]\label{def1}For any positive integer $k$, $\mu_k$ is called $k$-th singular point quantity of
the origin of system (\ref{S-complex}) and (\ref{S-real-0}). And
the origin of system (\ref{S-complex}) or (\ref{S-real-0}) is
called center if $\mu _k  =0, k = 1,2,3, \cdots $.
\end{defn}
In fact,  let $v_{2k+1}(2\pi)$ is $k$-th focal value of the origin
of system (\ref{S-real-0}), for the above each $\mu_k$, if
$\mu_0=\mu_1=\cdots=\mu_{k-1}=0,$ then
\begin{equation}\label{}
   \mu_k={\bf{i}}\, v_{2k+1}(2\pi).
\end{equation}

\begin{defn}[see \cite{Liu-Huang}]\label{def2}For any positive integer $k$, $\tau_k$ is called $k$-th period constant of the
origin of system (\ref{S-complex}) and (\ref{S-real-0}). And the
origin of system (\ref{S-complex}) or (\ref{S-real-0}) is called
isochronous center if $\tau_k=\mu _k  =0\,(or\, p_k=q_k=0), k = 1,2,3, \cdots $,
namely the two equtions\, $ \dot \xi=\xi$\,\,{\rm and}\,$\dot
\eta=-\eta$  hold  in (\ref{Russian2}). In here, we call $ p_k,\,q_k$ $k$-th isochronous constants of the
origin of system (\ref{S-complex}) and (\ref{S-real-0}).
\end{defn}

Now we recall the algorithm to compute period constants.
\begin{lem}[see \cite{Liu-Huang}]\label{Theorem F}
 For system (\ref{S-complex}),we can derive uniquely  the following formal series:
\begin{equation}\label{}
f(z,w) = z + \sum\limits_{k + j = 2}^\infty  {c'_{kj} z^k w^j }
,\,\, g(z,w) = w + \sum\limits_{k + j = 2}^\infty  {d'_{k,j} w^k
z^j } ,
\end{equation}
where $c'_{k+1,k}=d'_{k+1,k}=0,\: k=1,2,\cdots,$ such that
\begin{equation}\label{3.3}
\frac{{df }}{{dT}} = f(z,w)  + \sum\limits_{j = 1}^\infty  {p'_j\,
z ^{j + 1} w ^j },\,\,\frac{{dg }}{{dT}} =  - g(z,w)  -
\sum\limits_{j = 1}^\infty {q'_j \,w ^{j + 1} z ^j },
\end{equation}
and when  $k-j-1 \neq 0,$ $c'_{kj}$ and $d'_{kj}$ are determined,
when  $k-j-1 = 0$, $p'_j$ and $q'_j\,(j=1,2,\cdots)$  are
determined.
\end{lem}
The relations between $p_j,q_j$ and $p'_j,q'_j\,(j=1,2,\cdots)$
are the following lemma.
 \begin{lem}[see \cite{Liu-Huang}]\label{Theorem G}
Let $p_0=q_0=p'_0=q'_0=0$. If existing a positive integer $m$,
such that
\begin{equation}\label{p=q=0}
p_0=q_0=p_1=q_1=\cdots=p_{m-1}=q_{m-1}=0,
\end{equation}
then,
\begin{equation}\label{p=p'}
p'_0=q'_0=p'_1=q'_1=\cdots=p'_{m-1}=q'_{m-1}=0
,\,p_m=p'_m,\,q_m=q'_m.
\end{equation}
per contra, it holds as well.
\end{lem}

For differential systems (\ref{3w-1}), we apply the center
manifold theorem \cite{Carr}, the three-dimensional system
(\ref{3w-1}) has an approximation to the center manifold taking
the form
\begin{equation}\label{approx}
 u = u(x, y)=u_2(x, y)+{\rm h.o.t.}
 \end{equation}
where $u_2$ is a quadratic homogeneous polynomial in $x$ and $ y$,
and h.o.t denotes the terms with orders greater than or equal to
3. Substituting $u = u(x, y)$ into the equations of system
(\ref{3w-1}), we can obtain a generic two-dimensional differential
system with the same form as systems (\ref{S-real-0})
\begin{equation}\label{2w-lu}
\begin{array}{l} \dot x =-y+ {\rm h.o.t.},\;\; \dot y = x+ {\rm
h.o.t.}
\end{array}
\end{equation}
Usually the above system (\ref{2w-lu}) is called the reduced
equations of system (\ref{3w-1}), then by means of transformation
(\ref{2w-t1}), system (\ref{2w-lu}) can be changed into its
corresponding concomitant system with the same form as
 system (\ref{S-complex}),

 \begin{equation}\label{3w-2-1}
\begin{array}{l}
{{{\rm{d}}z} \over {{\rm{d}}T}} = z +{\rm h.o.t.}, \;\;\;
{{{\rm{d}}w} \over {{\rm{d}}T}} =  - w -{\rm h.o.t.}\;
 \end{array}
 \end{equation}

Furthermore, we have the following definitions.

\begin{defn}For the  system (\ref{3w-1}), the
origin  is called center on center manifolds if the origin of
system (\ref{2w-lu}) or (\ref{3w-2-1})  is a center, moreover, the
origin is called isochronous center on center manifolds if the
origin of system (\ref{2w-lu}) or (\ref{3w-2-1})  is isochronous
center.
\end{defn}
With the convenience, in the following we call the center (isochronous center) of system (\ref{3w-1}) on center manifolds as the center (isochronous center) of system (\ref{3w-1}).

In fact, by computing the singular point quantity  $\mu_k$ and
period constant $\tau_k$  of the origin of system (\ref{2w-lu})
and (\ref{3w-2-1}) according to the definitions \ref{def1} and
\ref{def2}, we can find the center conditions or isochronous
center conditions for systems (\ref{3w-1}) restricted to center
manifolds. However, one  know the dimensional reduction is not
necessarily for center-focus determining on center manifold, that
is to say, without obtaining its reduced system (\ref{2w-lu}) or
(\ref{3w-2-1}),  we can also calculate directly the corresponding
singular point quantity of system (\ref{3w-1}), see e.g.
\cite{Wang-2010,Du2014}. Correspondingly we can develop the
algorithm of Lemmas \ref{Theorem F} and \ref {Theorem G} to
investigate directly the isochronous problem on center manifolds
of the 3D systems (\ref{3w-1}), which will be seen in the next
section.
\section{Isochronous constants of 3D systems}
\hskip\parindent Now  we investigate the direct computational
method of period constants $\tau_k$ for the isochronous center on
center manifolds of  the 3D systems (\ref{3w-1}).  Firstly by
means of transformation (\ref{2w-t1}), systems (\ref{3w-1}) can
become following complex system
\begin{equation}\label{3w-2}
\begin{array}{l}
 \frac{{{\rm d}z}}
{{{\rm d}T}} = z + \sum\limits_{k + j + l = 2}^\infty  {a_{kjl} }
z^k
w^j u^l  = \tilde Z(z, w, u), \\
   \frac{{{\rm d}w}}
{{{\rm d}T}} =  - w - \sum\limits_{k + j + l = 2}^\infty  {b_{kjl}
w^k
z^j u^l }  =  -\tilde W(z, w, u), \\
   \frac{{{\rm d}u}}
{{{\rm d}T}} = {\bf{i}}du + \sum\limits_{k + j + l = 2}^\infty
{\tilde d_{kjl} } z^k w^j u^l  = \tilde U(z, w, u)
 \end{array}
 \end{equation}
where $z, w, T, a_{kjl}, b_{kjl}, \tilde d_{kjl}  \in {\mathbb
C}\, (k, j, l \in {{\mathbb N}})$, we also call that system
(\ref{3w-1}) and (\ref{3w-2}) are concomitant. When there exists
no misunderstanding, $\tilde d_{kjl}, \;\tilde Z, \;\tilde W$ and
$\;\tilde U$ are still written as $d_{kjl}, Z, W, U$.

Similar to the result of Lemma \ref{Theorem F}, we can obtain the
following theorems.
\begin{thm}\label{wql-1} For system (\ref{3w-2}),
 when taking $c_{100}  = 1, c_{001} = c_{010}=0, c_{k+1,k,0}  = 0, k = 1, 2, \cdots$,
  we can derive successively and uniquely the terms of the
 following formal series:
 \begin{equation}\label{3w-t1}
\begin{array}{l}
 f(z, w, u) = z +
\sum\limits_{\alpha  + \beta  + \gamma  = 2}^\infty  {c_{\alpha
\beta \gamma } z^\alpha  w^\beta  } u^\gamma
 \end{array}
 \end{equation}
 such that
\begin{equation}\label{3w-t2}
\frac{{\rm{d}}f}{{\rm{d}}T}-f = \frac{{\partial f}} {{\partial
z}}Z - \frac{{\partial f}} {{\partial w}}W + \frac{{\partial f}}
{{\partial u}}U-f =z \cdot \sum\limits_{m = 1}^\infty {p'_m
(zw)^{m} }
\end{equation}
and if $\alpha \ne \beta+1$ or $\alpha=\beta+1, \gamma \ne 0$,
$c_{\alpha \beta \gamma }$ is determined by following recursive
formula:
\begin{equation}\label{3w-gs1}
\begin{array}{l}
 c_{\alpha \beta \gamma }  = \frac{1}
{{1+\beta - \alpha - {\bf{i}}d\gamma }}\times \\\\ \quad \sum\limits_{k +
j + l = 3}^{\alpha + \beta + \gamma  + 2} {[(\alpha  - k + 1)a_{k,
j - 1, l} - (\beta - j + 1)b_{j, k - 1, l}  + (\gamma  - l)d_{k -
1, j - 1, l + 1} ]}\\\quad\quad\quad\quad \quad \cdot c_{\alpha -
k + 1, \beta  - j + 1, \gamma -l}
\end{array}
\end{equation}
and for any positive integer $m, \, \, p'_m $ is determined by
following recursive formula:
\begin{equation}\label{3w-gs2}
\begin{array}{l}
p'_m  = \sum\limits_{k + j + l = 3}^{2m + 3} {[(m - k + 2)a_{k,j -
1,l} - (m - j + 1)b_{j,k - 1,l}  - l\,d_{k - 1,j - 1,l + 1} ]}
\\ \quad\quad\quad\quad \quad \cdot c_{m - k + 2,m - j + 1, - l}
\end{array}
\end{equation}
and when  $\alpha < 0$ or $\beta < 0$ or $\gamma<0$ or $\gamma=0,
\alpha=\beta+1$, we have let $\;c_{\alpha, \beta, \gamma} = 0$.
\end{thm}
\begin{proof}
From system (\ref{3w-2}), we can denote
\begin{equation}\label{}
\begin{array}{l}
  Z = z + \sum\limits_{k + j + l \ge 3} {a_{k, j - 1, l} z^k w^{j - 1} } u^l , \\
   W = w + \sum\limits_{k + j + l \ge 3} {b_{j, k - 1, l} z^{k - 1} w^j u^l
   },   \\
   U={\bf i}du + \sum\limits_{k + j + l \ge 3} {d_{k - 1, j - 1, l + 1} z^{k - 1} w^{j - 1} } u^{l + 1}
 \end{array}
 \end{equation}
 then we have the following,
 $$
\begin{array}[t]{l}
  {{\partial f} \over {\partial z}}Z - {{\partial f} \over {\partial w}}W + {{\partial f} \over {\partial
  u}}U-f
  \\
 \quad\quad  = \sum\limits_{\alpha  + \beta  + \gamma  \ge 1} {(\alpha  - \beta  + {\bf{i}}d\gamma -1)c_{\alpha \beta \gamma } z^\alpha  w^\beta  u^\gamma  }  +
   \\
 \quad\quad \quad \sum\limits_{\alpha  + \beta  + \gamma  \ge 1} {\sum\limits_{k + j + l \ge 3} {[\alpha a_{k, j - 1, l}  - \beta b_{j, k - 1, l}  + \gamma d_{k - 1, j - 1, l + 1} ]c_{\alpha \beta \gamma } z^{\alpha  + k - 1} w^{\beta  + j - 1} u^{\gamma  + l} } }
  \\
   \quad\quad= \sum\limits_{\alpha  + \beta  \ge 1} {z^\alpha  w^\beta  u^\gamma  \{(\alpha  - \beta  + {\bf{i}}d\gamma -1)c_{\alpha \beta \gamma }
   }+
   \\
  \quad\quad   \sum\limits_{k + j + l \ge 3} {[(\alpha  - k + 1)a_{k, j - 1, l}  - (\beta  - j + 1)b_{j, k - 1, l}  + (\gamma  - l)d_{k - 1, j - 1, l + 1} ]c_{\alpha  - k + 1, \beta  - j + 1, \gamma  - l}
  }\}
\end{array}$$
and comparing the above power series with the right side of
(\ref{3w-t2}), then we can obtain the  recursive formulas
(\ref{3w-gs1}) and (\ref{3w-gs2}).
\end{proof}

With the same principle, we have

\begin{thm}\label{wql-1-1} For system (\ref{3w-2}),
 when taking $e_{100}  = 1, e_{001} = e_{010}=0, e_{k+1,k,0}  = 0, k = 1, 2, \cdots$,
  we can derive successively and uniquely the terms of the
 following formal series:
 \begin{equation}\label{3w-t1-1}
\begin{array}{l}
 g(w, z, u) = w +
\sum\limits_{\alpha  + \beta  + \gamma  = 2}^\infty  {e_{\alpha
\beta \gamma } w^\alpha  z^\beta  } u^\gamma
 \end{array}
 \end{equation}
 such that
\begin{equation}\label{3w-t2-1}
\frac{{\rm{d}}g}{{\rm{d}}T} +g= \frac{{\partial g}} {{\partial
z}}Z - \frac{{\partial g}} {{\partial w}}W + \frac{{\partial g}}
{{\partial u}}U+g =-w \cdot \sum\limits_{m = 1}^\infty {q'_m
(zw)^{m} }
\end{equation}
and if $\alpha \ne \beta+1$ or $\alpha=\beta+1, \gamma \ne 0$,
$e_{\alpha \beta \gamma }$ is determined by following recursive
formula:
\begin{equation}\label{3w-gs1-1}
\begin{array}{l}
 e_{\alpha \beta \gamma }  = \frac{1}
{{1+\beta - \alpha - {\bf{i}}d\gamma }}\\\\ \quad \sum\limits_{k +
j + l = 3}^{\alpha + \beta + \gamma  + 2} {[(\alpha  - k + 1)b_{k,
j - 1, l} - (\beta - j + 1)a_{j, k - 1, l}  - (\gamma  - l)d_{j -
1, k - 1, l + 1} ]}\\\quad\quad\quad\quad \quad {\cdot} e_{\alpha
- k + 1, \beta  - j + 1, \gamma -l}
\end{array}
\end{equation}
and for any positive integer $m, \, \, p'_m $ is determined by
following recursive formula:
\begin{equation}\label{3w-gs2-1}
\begin{array}{l}
p'_m  = \sum\limits_{k + j + l = 3}^{2m + 3} {[(m - k + 2)b_{k,j -
1,l} - (m - j + 1)a_{j,k - 1,l}  - l\,d_{j - 1,k - 1,l + 1} ]}\\
\quad\quad\quad\quad\quad\quad  {\cdot} e_{m - k + 2,m - j + 1, -
l}
\end{array}
\end{equation}
and when  $\alpha < 0$ or $\beta < 0$ or $\gamma<0$ or $\gamma=0,
\alpha=\beta+1$, we have let $\;c_{\alpha, \beta, \gamma} = 0$.
\end{thm}

Furthermore we can know that system (\ref{3w-2-1}),  as the
concomitant system of reduced equations of the original 3D systems
(\ref{3w-1}), has also $p_j, q_j$ series which can be uniquely
determined in the normal form of the lemma \ref{lem A}. Thus we
have the relations between $p_j, q_j$ and $p'_j,
q'_j\,(j=1,2,\cdots)$ in theorems \ref{wql-1} and \ref{wql-1-1}
are as follows:
\begin{thm}\label{relation pj p'j}
Let $p_0=q_0=p'_0=q'_0=0$. If existing a positive integer $m$,
such that
\begin{equation}\label{p=q=0}
p_0=q_0=p_1=q_1=\cdots=p_{m-1}=q_{m-1}=0,
\end{equation}
then,
\begin{equation}\label{p=p'}
p'_0=q'_0=p'_1=q'_1=\cdots=p'_{m-1}=q'_{m-1}=0
,\,p_m=p'_m,\,q_m=q'_m.
\end{equation}
per contra, it holds as well. Correspondingly, the origin of system (\ref{3w-1}) is an isochronous center if and only if $p'_k=q'_k=0,\,k=1,2,3\cdots.$ We also call $p'_k,\,q'_k$ the isochronous constants of the origin of system (\ref{3w-1}).
\end{thm}

\begin{proof}
On the one hand by substituting the center manifold
(\ref{approx}): $u=u(x, y)$, then by means of transformation
(\ref{2w-t1}), we can obtain systems (\ref{2w-lu}) and
(\ref{3w-2-1}) in turn from system (\ref{3w-1}). On the other
hand, first by means of transformation (\ref{2w-t1}), then by
substituting the center manifold (\ref{approx}) with the following
form
\begin{equation}\label{approx-1}
 u = u(x, y)= u({{z+w}\over{2}}, {{z-w}\over{2{\bf i}}})\;\buildrel \Delta \over =\;\tilde u(z, w)
 \end{equation}
 we can also obtain systems
(\ref{3w-2}) and (\ref{3w-2-1}) in turn from system (\ref{3w-1}).

 Thus from the above
relation (\ref{approx-1}) and the theorem \ref{wql-1}, we have
$$
\begin{array}[t]{l}
f(z, w, u) = f(z, w, \tilde u(z, w)) = z + \sum\limits_{\alpha  +
\beta + \gamma  = 2}^\infty  {c_{\alpha \beta \gamma } z^\alpha
w^\beta }
[\tilde u(z, w)]^\gamma  \\
\quad \quad \quad \quad\,\,\, \buildrel \Delta \over =  \tilde
f(z, w) +{\rm{h.o.t.}}
\end{array} $$
where
\begin{equation}\label{f-new}
\tilde f(z,w)=z+\sum\limits_{k+j=2}^{2n+1}c'_{kj} z^k w^j
\end{equation}
where the coefficients $c'_{kj}$ in series (\ref{f-new})  are
determined successively and uniquely, and for $n$ being any a
positive integer, $c'_{k+1,k}=0,\: k=1,2,\cdots, n$. Moreover we
have
\begin{equation}\label{3w-t4}
\begin{array}{l}
 \left. {{{df} \over {dT}}} \right|_{(\ref{3w-2})}-f={{\partial f}
\over {\partial z}}Z - {{\partial f} \over {\partial y}}W +
{{\partial f} \over {\partial u}}U -f= \left. {{{d\tilde f} \over
{dT}}} \right|_{(\ref{3w-2-1})}-\tilde f  + {\rm{h.o.t.}} \\\\
\quad \quad \quad \;\; = z \cdot \sum\limits_{m = 1}^{n} {p'_m
(zw)^{m}} + {\rm{h.o.t.}}
 \end{array}
 \end{equation}

Next, according to the lemma \ref{lem A},  for system
(\ref{3w-2-1}), we can uniquely determine the following series
\begin{equation}\label{}
\xi(z,w)=z+\sum\limits_{k+j=2}^\infty c_{kj} z^k w^j \;\buildrel
\Delta \over =\;  \tilde \xi(z, w) + {\rm{h.o.t.}}.
\end{equation}
where
\begin{equation}\label{xi-new}
\tilde \xi(z,w)=z+\sum\limits_{k+j=2}^{2n+1}c_{kj} z^k w^j
\end{equation}
and $c_{k+1,k}=0,\: k=1,2,\cdots$, such that
\begin{equation}\label{3.16}
\begin{array}{l}
  \frac{d\xi}{dT}-\xi= \sum\limits_{m = 1}^\infty  {p_m\,
\xi ^{m + 1} \eta ^m }=\sum\limits_{m = 1}^\infty  {p_m(
  z^{m+1}w^m+{\rm{h.o.t.}})}
  \\\\
\quad \quad \quad \;=\sum\limits_{m = 1}^n  {p_m(
  z^{m+1}w^m+{\rm{h.o.t.}})}+\rm{h.o.t.}
   \end{array}
  \end{equation}
Considering the uniqueness of formal series one term by one term
in the lemma \ref{lem A} and theorem \ref{wql-1}, from expressions
(\ref{3.16}) and (\ref{3w-t4}), with mathematical induction for
$m$, it is easy to get $ p'_0=p'_1=\cdots=p'_{m-1}=0 ,\,p_m=p'_m$,
and  $\tilde \xi (z, w) = \tilde f(z, w)$ with mathematical
induction. In the same way, it is also easy to get  $
q'_0=q'_1=\cdots=q'_{m-1}=0$ and $q_m=q'_m$, namely the expression
(\ref{p=p'}) holds.
\end{proof}

For the isochronous center on center manifolds of 3D system
(\ref{3w-1}), Theorem \ref{wql-1}, Theorem \ref{wql-1-1} and
\ref{relation pj p'j} give a direct algorithm to compute period
constant $\tau_m=p_m+q_m$, namely, we can apply directly the above
theorems to find the necessary conditions of the isochronous
center, needless to solve firstly the center problem. The
algorithm is linear recurrence and then avoids complex integrating
operations and solving equations, which is easy to realize with
computer algebra system.

\begin{rmk}
we can't use Theorem \ref{wql-1}, Theorem \ref{wql-1-1} and
Theorem \ref{relation pj p'j} to compute singular point quantities
$\mu_l=p_l-q_l\,\,(l=1,2,\cdots)$, because the one condition of
Theorem \ref{relation pj p'j} is expressions (\ref{p=q=0}), while
the computation of $\mu_l$ is only under the condition
$\mu_0=\mu_1=\cdots=\mu_{l-1}=0$.
\end{rmk}
\section{Examples }
 \hskip\parindent As the applications,  now we consider the  isochronous center of two quadratic
 systems restricted to the center manifold. In the following, we write isochronous constants $p'_k,\,q'_k$ in Theorem \ref{wql-1} and \ref{wql-1-1} as $p_k,\,q_k$ respectively.
\subsection{The Moon-Rand system}
\hskip\parindent The Moon-Rand system was introduced by Moon and
Rand and  developed to model the control of flexible space
structures in \cite{Moon-Rand}, which is a three dimensional
differential system with the following form:
\begin{equation}\label{exam-1}
 \begin{array}{l}
 \frac{{{\rm{d}}x}}{{{\rm{d}}t}} = y,\quad
\frac{{{\rm{d}}y}}{{{\rm{d}}t}} =  - x - xu,\quad
\frac{{{\rm{d}}u}}{{{\rm{d}}t}} =  - c_0 u + c_1 x^2  + c_2 xy +
c_3 y^2
 \end{array}
\end{equation}
where $c_0,c_1,c_2,c_3 \in {{\mathbb R}}$  and $c_0>0$. Recently
some authors have studied its integrability and bifurcation of
limit cycles(e.g. see \cite{Mahdi-MR,Llibre-MR}), here we consider
its isochronicity on the local center manifold.  By means of
transformation: $z=x +{\bf{i}}y,\, w=x -{\bf{i}}y, \,T
={\bf{i}}\,t$,  we can get its complex concomitant system from
system (\ref{exam-1}):
\begin{equation}\label{MR-example}
\begin{array}{l}
 \frac{{{\rm{d}}z}}{{{\rm{d}}T}} = z + { {1 \over 2}}u(z + w),\\
 \frac{{{\rm{d}}w}}{{{\rm{d}}T}} =  - w - { {1 \over 2}}u(z + w), \\
 \frac{{{\rm{d}}u}}{{{\rm{d}}T}} =  - {\bf{i}}c_0 u + { {1 \over 4}}{\bf{i}}(c_1  - c_3  - {\rm{i}}c_2 )z^2  + {{1 \over 2}}{\bf{i}}(c_1  + c_3 )zw + {{1 \over 4}}{\bf{i}}(c_1  - c_3  + {\rm{i}}c_2
 )w^2.
\end{array}
\end{equation}
 For the system (\ref{MR-example}), according to
Theorem \ref{wql-1},Theorem \ref{wql-1-1} and Theorem
\ref{relation pj p'j}, we can get the recursive formulas to
compute isochronous constants for any positive integer  $m$,
for example when $m=20$, we can obtain the first twenty
isochronous constants as follows:
\begin{equation}\label{uu3}
\begin{array}{l}
p_1  =\frac{{c_0 (3c_1  + c_3 ) - {\bf{i}}(4c_1  + c_0 c_2  + 4c_3
)}}{{8c_0(c_0  - 2{\bf{i}})}},\;  p_2=- \frac{{3(c_0  +
2{\bf{i}})c_2^2 }}{{1024(c_0  - 2{\bf{i}})}},\;\\
p_3=\cdots=p_{20}=0,\\
 q_1  =-\frac{{c_0 (3c_1  + c_3 ) + {\bf{i}}(4c_1  + c_0 c_2  + 4c_3
)}}{{8c_0(c_0 + 2{\bf{i}})}},\; q_2= \frac{{3(c_0  -
2{\bf{i}})c_2^2 }}{{1024(c_0+ 2{\bf{i}})}}, \;\\
q_3=\cdots=q_{20}=0
\end{array}
\end{equation}
where for each $p_m$ in the above expression, we have already
let\,$p_1=p_2=\cdots=p_{l-1}=0$; for each $q_m$, we have already
let\,$p_1=q_1=p_2=q_2=\cdots=p_{l-1}=q_{l-1}=0,l=2,3,\cdots 20$.

According to the above calculating results, then we have

\begin{thm}\label{cen-condition00} The origin of system (\ref{exam-1})(or (\ref{MR-example})) is an isochronous center   if and only if
the following conditions is satisfied:
\begin{equation}\label{iso-condition-0}
  \begin{array}{ll}
c_1=c_2=c_3 = 0
\end{array}
\end{equation}
\end{thm}
\begin{proof}

The necessity of condition is obvious, we can obtain easily the
above conditions from the vanishing of the first 20 isochronous constants, namely,
let\,$p_1=q_1=\cdots=p_{m}=q_{m}=0,m=2,3,\cdots, 20$.

Now we prove the sufficient condition, this technique derives from
the Darboux theory of integrability (one can see some notions and
facts in \cite{zhangx-09}-\cite{Cunha2}). In fact, when
$c_1=c_2=c_3 = 0$, we figure out easily the algebraic invariant
surface for system (\ref{MR-example}): $F(z,w,u)=u$. One can
observe that $F(z,w,u)=u=0$ is just the center eigenspace, i.e.,
$(x,y)$-plane. Thus it forms a local center manifold in a
neighborhood of the origin.  We substitute $u=0$ into the first
and second equations of the system defined by system
(\ref{MR-example}), we have the differential equations
\begin{equation}\label{MR-example-1}
\begin{array}{l}
\frac{{{\rm{d}}z}}{{{\rm{d}}T}} = z,\;\;
\frac{{{\rm{d}}w}}{{{\rm{d}}T}} =  - w
\end{array}
\end{equation}
then the origin is a isochronous center for systems
(\ref{MR-example-1}), thus  when $c_1=c_2=c_3 = 0$, the origin is
a isochronous center for the flow of system (\ref{MR-example}) or
(\ref{exam-1}) restricted to a center manifold. We complete the
proof of this theorem.
\end{proof}

\subsection{A class of complex quadratic system}
\hskip\parindent  A class of complex quadratic system  with the
following form is considered
\begin{equation}\label{C-example}
\begin{array}{l}
\frac{{{\rm{d}}z}}{{{\rm{d}}T}} = z(1 + a_1 z + b_1 w + c_1 u),\\
\frac{{{\rm{d}}w}}{{{\rm{d}}T}} =  - w(1 + a_2 z + b_2 w + c_2
u),\\ \frac{{{\rm{d}}u}}{{{\rm{d}}T}} = u({\bf{i}} \,r + a_3 z +
b_3 w + c_3 u)
\end{array}
\end{equation}
where $u,r \in {{\mathbb R}}$, $z, w, T \in {{\mathbb C}}$, and
\begin{equation}\label{}
 \begin{array}{l}
a_2  = \bar b_1 ,b_2  = \bar a_1 ,c_2  = \bar c_1,
 b_3  =  - \bar a_3 ,c_3  =  - \bar c_3.
 \end{array}
\end{equation}
In fact, by means of transformation: $z=x +{\bf{i}}y,\, w=x
-{\bf{i}}y, \,T ={\bf{i}}\,t$,  we can get its real concomitant
system from system (\ref{C-example}):
 \begin{equation}\label{exam-2}
 \begin{array}{l}
 \frac{{{\rm{d}}x}}{{{\rm{d}}t}} = -y  + X_2 (x, y, u) = X, \\
 \frac{{{\rm{d}}y }}{{{\rm{d}}t}} =x  + Y_2 (x, y, u) = Y, \\
 \frac{{{\rm{d}}u}}{{{\rm{d}}t}} ={-r}\,u + U_2 (x, y, u) = U \\
 \end{array}
\end{equation}
where $X_2,Y_2$ and $U_2$ are all quadratic homogeneous
polynomials in $(x, y, u)$ determined by the coefficients of
system (\ref{C-example}).

 For the system (\ref{C-example}), by applying
the same algorithm, we can get the recursive formulas to compute
isochronous constants for any positive integer $m$, for
example when $m=20$, we can obtain the first twenty generalized
constants as follows:
\begin{equation}\label{uu3}
\begin{array}{l}
p_1  =- b_1 (a_1  + a_2 ),  p_2=p_3=\cdots=p_{20}=0,\\
 q_1  =a_2 (b_1  + b_2 ), q_2=q_3=\cdots=q_{20}=0,
\end{array}
\end{equation}
where for each $p_m$ in the above expression, we have already
let\,$p_1=p_2=\cdots=p_{l-1}=0$; for each $q_m$, we have already
let\,$p_1=q_1=p_2=q_2=\cdots=p_{l-1}=q_{l-1}=0,l=2,3,\cdots 20$.

According to the above calculating results, then we have

\begin{thm}\label{cen-condition00} The origin of system (\ref{exam-2}) (or \ref{C-example})) is an isochronous center  if and only if
the following conditions is satisfied:
\begin{equation}\label{}
  \begin{array}{ll}
b_1 (a_1  + a_2 )=0, \,\,a_2 (b_1  + b_2 ) = 0
\end{array}
\end{equation}
namely one of  the following four conditions holds:
\begin{equation}\label{iso-condition}
  \begin{array}{ll}
{\,\rm (i)\,}b_1=b_2=0,{\;\;\;\;(ii)\,}a_1=a_2=0,{\;\;\;\;(iii)\,}
b_1=a_2=0,{\;\;\;\;(iv)\,} a_2=-a_1,b_1=-b_2.
\end{array}
\end{equation}
\end{thm}

\begin{proof}

The necessity of condition is obvious by
letting\,$p_1=q_1=\cdots=p_{m}=q_{m}=0,m=2,3,\cdots, 20$.

Now we prove the sufficient condition. In fact,  we can figure out
easily one algebraic invariant surface for system
(\ref{C-example}): $F(z,w,u)=u$, namely there exists a polynomial
$K(z,w,u)={\bf{i}} \,r + a_3 z + b_3 w + c_3 u$, as the cofactor
of $F(z,w,u)$, such that $\left. {{{{\rm{d}}F} \over {{\rm{d}}t}}}
\right|_{(\ref{C-example})}= KF$. One can observe that
$F(z,w,u)=u=0$ is just the center eigenspace, i.e., $(x,y)$-plane in system (\ref{exam-2}).
Thus it forms a local center manifold in a neighborhood of the
origin.  We substitute $u=0$ into the first and second equations
of the system defined by system (\ref{C-example}), we have the
differential equations
\begin{equation}\label{C-example-1}
\begin{array}{l}
\frac{{{\rm{d}}z}}{{{\rm{d}}T}} = z(1 + a_1 z + b_1 w ),\;\;
\frac{{{\rm{d}}w}}{{{\rm{d}}T}} =  - w(1 + a_2 z + b_2 w )
\end{array}
\end{equation}

Case (i): if $b_1=b_2=0$ in the conditions (\ref{iso-condition})
holds, then system (\ref{C-example-1}) has the corresponding form
as follows
\begin{equation}\label{C-example-1i}
\begin{array}{l}
\frac{{{\rm{d}}z}}{{{\rm{d}}T}} = z(1 + a_1 z ),\;\;
\frac{{{\rm{d}}w}}{{{\rm{d}}T}} =  - w(1 + a_2 z ).
\end{array}
\end{equation}
Furthermore, there exists a linear change of coordinates: $(z,w)=
(\xi (1 - a_1 \xi )^{ - 1}, \eta(1 - a_1 \xi
)^{\frac{{a\!_2}}{{a\!_1 }}} )$, which transforms system
(\ref{C-example-1i}) into the following form:
\begin{equation}\label{simp-1}
\frac{{{\rm{d}}\xi}}{{{\rm{d}}T}} =\xi,\;\;\;\;
\frac{{{\rm{d}}\eta}}{{{\rm{d}}T}} =  -\eta
\end{equation}
then the origin is a isochronous center for systems
(\ref{C-example-1i}).

Case (ii): if $a_1=a_2=0$ in the conditions (\ref{iso-condition})
holds, then system (\ref{C-example-1}) has the corresponding form
as follows
\begin{equation}\label{C-example-1ii}
\begin{array}{l}
\frac{{{\rm{d}}z}}{{{\rm{d}}T}} = z(1 + b_1 w ),\;\;
\frac{{{\rm{d}}w}}{{{\rm{d}}T}} =  - w(1 + b_2w ).
\end{array}
\end{equation}
Also, there exists a linear change of coordinates: $(z,w)= (\xi (1
- b_2 \eta )^{\frac{{b\!_1}}{{b\!_2 }}}, \eta (1 -b_2 \eta )^{ -
1})$, which transforms system (\ref{C-example-1ii}) into the
 form of (\ref{simp-1}), then the origin is a isochronous center for systems
(\ref{C-example-1ii}).

Case (iii): if $b_1=a_2=0$ in the conditions (\ref{iso-condition})
holds, then system (\ref{C-example-1}) has the corresponding form
as follows
\begin{equation}\label{C-example-1iii}
\begin{array}{l}
\frac{{{\rm{d}}z}}{{{\rm{d}}T}} = z(1 + a_1 z ),\;\;
\frac{{{\rm{d}}w}}{{{\rm{d}}T}} =  - w(1 + b_2w ).
\end{array}
\end{equation}
Also, there exists a linear change of coordinates: $(z,w)= (\xi (1
- a_1 \xi )^{-1}, \eta (1 -b_2 \eta )^{ - 1})$, which transforms
system (\ref{C-example-1iii}) into the
 form of (\ref{simp-1}), then the origin is a isochronous center for systems
(\ref{C-example-1iii}).

Case (iv): if $ a_2=-a_1,b_1=-b_2$ in the conditions
(\ref{iso-condition}) hold, then system (\ref{C-example-1}) has
the corresponding form as follows
\begin{equation}\label{C-example-1iv}
\begin{array}{l}
\frac{{{\rm{d}}z}}{{{\rm{d}}T}} = z(1 + a_1 z +b_1w),\;\;
\frac{{{\rm{d}}w}}{{{\rm{d}}T}} =  - w(1-a_1 z - b_1w ).
\end{array}
\end{equation}
Also, there exists a linear change of coordinates: $(z,w)= (\xi (1
- a_1 \xi+b_1\eta )^{-1}, \eta (1 - a_1 \xi+b_1\eta )^{ - 1})$,
which transforms system (\ref{C-example-1iv}) into the
 form of (\ref{simp-1}), then the origin is a isochronous center for systems
(\ref{C-example-1iv}).

Therefore when one of the four conditions (\ref{iso-condition})
holds, the origin is a isochronous center for the flow of system
(\ref{C-example}) or (\ref{exam-2})  restricted to a center
manifold. We complete the proof of this theorem.
\end{proof}



\begin{rmk}
The algorithm involved in  Theorem \ref{wql-1}, Theorem
\ref{wql-1-1} and Theorem \ref{relation pj p'j}  gives a available
method to find the necessary conditions of isochronous center for
the 3D system (\ref{3w-1}) restricted to a center manifold.
However, the proof of sufficient conditions is still a difficult
question except a few cases.
\end{rmk}

\subsection*{Acknowledgements}
\hskip \parindent This work was supported by Natural Science
Foundation of China grants 11461021 and Nature
Science Foundation of Guangxi grant No.2016GXNSFDA380031.

\end{document}